\documentclass{article}
\usepackage[english]{babel}
\usepackage{graphicx} 
\usepackage{enumerate}
\usepackage[utf8]{inputenc}

\usepackage{packages}
\usepackage{commands}

\title{Quotient graphs and stochastic matrices}
\author{Frederico Cançado\footnote{frederico.cancado@dcc.ufmg.br.}, Gabriel Coutinho\footnote{gabriel@dcc.ufmg.br.}}

\newtheorem{theorem}{Theorem}
\newtheorem{lemma}[theorem]{Lemma}
\newtheorem{corollary}[theorem]{Corollary}
\newtheorem{proposition}[theorem]{Proposition}

\newtheorem{theorem2}{Theorem}

\begin{document}

\AtEndDocument{%
  \par
  \medskip
  \begin{tabular}{@{}l@{}}%
    \textsc{Gabriel Coutinho}\\
    \textsc{Dept. of Computer Science} \\ 
    \textsc{Universidade Federal de Minas Gerais, Brazil} \\
    \textit{E-mail address}: \texttt{gabriel@dcc.ufmg.br} \\ \ \\
    \textsc{Frederico Cançado} \\
    \textsc{Dept. of Computer Science} \\ 
    \textsc{Universidade Federal de Minas Gerais, Brazil} \\
    \textit{E-mail address}: \texttt{frederico.cancado@dcc.ufmg.br} \\ 
  \end{tabular}}

\maketitle

\begin{abstract}
    Whenever graphs admit equitable partitions, their quotient graphs highlight the structure evidenced by the partition. It is therefore very natural to ask what can be said about two graphs that have the same quotient according to certain equitable partitions. This question has been connected to the theory of fractional isomorphisms and covers of graphs in well-known results that we briefly survey in this paper. We then depart to develop theory of what happens when the two graphs have the same symmetrized quotient, proving a structural result connecting this with the existence of certain doubly stochastic matrices. We apply this theorem to derive a new characterization of when two graphs have the same combinatorial quotient, and we also study graphs with weighted vertices and the related concept of pseudo-equitable partitions. Our results connect to known old and recent results, and are naturally applicable to study quantum walks.
\end{abstract}


\section{Introduction} \label{sec:intro}

Let $G$ and $H$ be graphs on the same number of vertices, with adjacency matrices $A_G$ and $A_H$. The linear programming relaxation of the graph isomorphism problem consists of searching for a matrix $M$ so that
\begin{equation}
	A_G M = M A_H \quad \text{and} \quad \text{$M$ is doubly stochastic.} \label{eq:fracisodef}
\end{equation}
Note that if $M$ is required to have integer entries, then any such $M$ is a permutation matrix, and therefore describes an isomorphism between $G$ and $H$. We say the graphs $G$ and $H$ are \textsl{fractionally isomorphic} if there exists a matrix $M$ satisfying \eqref{eq:fracisodef}.

A well-known theorem \cite[Chapter 6]{scheinerman2013fractional} provides equivalent characterizations of fractionally isomorphic graphs.

\begin{theorem}[Theorem 6.5.1 in \cite{scheinerman2013fractional}] \label{thm1}
    For two graphs $G$ and $H$ on the same number of vertices, the following are equivalent.
    \begin{enumerate}[(i)]
        \item They are fractionally isomorphic, that is, there is a doubly-stochastic matrix $M$ such that $A_G M=M A_H$. \label{thm1-1}
        \item $G$ and $H$ have some common equitable partition. \label{thm1-2} 
        \item $G$ and $H$ have in common the coarsest equitable partition. \label{thm1-3}
        \item $D(G)=D(H)$. \label{thm1-4}
    \end{enumerate}
\end{theorem}

The equivalence between \eqref{thm1-1} and \eqref{thm1-2} was proved in \cite{Ramana1994}, and we explain it in details below. We overview the remaining items along with other details in Section~\ref{sec:fraciso}. 

Given a partition $\pi = \{C_1,\dots,C_k\}$ of the vertex set of graph, let $P$ denote its characteristic matrix, that is, the $01$ matrix with rows indexed by vertices and $k$ columns, each equal to $\one_{C_i}$, the indicator vector of the set $C_i$. 

The partition $\pi$ is called \textsl{equitable} if the colspace of $P$ is $A_G$-invariant. Equivalently (see \cite{Compact_graphs_and_equitable_partitions}), $\pi$ is equitable if for any $v \in C_i$, the number of neighbours of $v$ in $C_j$ depends only on $i$ and $j$, thus denoted simply by $p_{ij}$. The directed graph with the $k$ cells of $\pi$ as its vertices and an arc of weight $p_{ij}$ from vertex $C_i$ to $C_j$ is denoted by $G/\pi$, and is referred to as the \textsl{quotient graph} of $G$ by $\pi$. Note that if $\pi$ is equitable, then (see \cite[Chapter 9]{godsil_algebraic_graph_theory}).
\[
	A_G P = P A_{G/\pi} \quad \text{and} \quad A_{G/\pi} = (P^\T P)^{-1} (P^\T A_G P).
\]
A third equivalent definition comes from the notion of \textsl{symmetrized quotient graph}, which is key to our paper. Given the characteristic matrix $P$, let $\wt{P}$ denote the matrix obtained from $P$ upon normalizing its columns, that is,
\[
	\wt{P} = P (P^\T P)^{-1/2}.
\]
We define $\wt{G/\pi}$ to be the weighted (unidirected) graph whose adjacency matrix is given by
\begin{equation}
	A_{\wt{G/\pi}} = \wt{P}^\T A_G \wt{P}, \quad \text{and thus} \quad A_G \wt{P} = \wt{P} A_{\wt{G/\pi}}, \label{eq:sympartition}
\end{equation}
so it follows that $\pi$ is equitable if and only if $A_G$ commutes with $\wt{P}\wt{P}^\T$ (again, see \cite[Lemma 1.1]{Compact_graphs_and_equitable_partitions}).

\begin{figure}[h]
	\begin{tikzpicture}
    \draw (2,1.73) node[draw,circle](q1){};
    \draw (1,0) node[draw,circle,fill=black](q2){};
    \draw (3,0) node[draw,circle,fill=black](q3){};
    \draw[-] (q1) edge node{} (q2);
    \draw[-] (q2) edge node{} (q3);
    \draw[-] (q3) edge node{} (q1);
    
    \draw (0,0.865) node[](r1){{$G=$}};
    \draw (4,0.865) node[](r2){{$G / \pi =$}};
    
    \draw (6,1.73) node[draw,circle](1){};
    \draw (6,0) node[draw,circle,fill=black](2){};
    \draw[->] (1) edge[bend left] node[right]{$2$} (2);
    \draw[->] (2) edge[bend left] node[left]{$1$} (1);
    \draw[-] (2) edge[loop below, below] node{1} (2);

    \draw (8,0.865) node[](r2){{$\wt{G/\pi} =$}};
    
    \draw (10,1.73) node[draw,circle](1){};
    \draw (10,0) node[draw,circle,fill=black](2){};
    \draw[-] (1) edge[right] node{$\sqrt{2}$} (2);
    \draw[-] (2) edge[loop below, below] node{1} (2);
\end{tikzpicture} \caption{Example of a graph $G$ with an equitable partition $\pi$ (white and black vertices), the quotient graph, and the symmetrized quotient graph.} 
\end{figure}
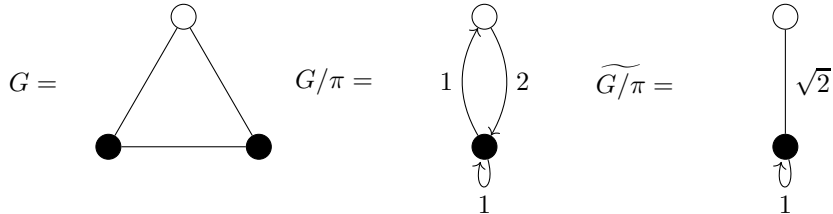

Equitable partitions and graph quotients were likely studied for the first time in \cite{first_reference_to_equitable}. We briefly mention two recent applications of graph quotients to completely different fields: the combinatorial quotient graphs, for instance, have been applied to facilitate the solution of linear programs in \cite{equitable_partition_and_linearoptimization}; and the symmetrized quotients have been used to study perfect state transfer --- a quantum walk phenomenon with applications in quantum computing \cite{PST_and_quantum_computr,kay2018perfect}.

To say that two graphs $G$ and $H$ have a \textsl{common equitable partition}, in the sense of Theorem~\ref{thm1}, is to say that there are equitable partitions of $G$ and $H$ with characteristic matrices $P$ and $Q$ respectively so that
\[
P^\T P = Q^\T Q , \quad \text{and} \quad P^\T A_G P = Q^\T A_H Q.
\]
(Note that we need not worry about using equality rather than some isomorphism sign, as the columns of each partition matrix may be freely reordered). Naturally the first condition implies both graphs have the same number of vertices, and together they imply that the quotient graphs are isomorphic. 

Note that both conditions imply that $A_{G/\pi} = A_{H/\sigma}$, but the converse is not necessarily true. A further observation is that
\[
	A_{G/\pi} = A_{H/\sigma} \implies A_{\wt{G/\pi}} = A_{\wt{H/\sigma}},
\]
as the procedure to symmetrize a matrix by diagonal similarity yields a unique solution. We examine this latter property and we provide an equivalent condition that goes in spirit of the equivalence between \eqref{thm1-1} and \eqref{thm1-2} in Theorem~\ref{thm1}. Our first main result is the following:

\begin{theorem}\label{thm:mainthm1} 
    Let $G$ and $H$ be graphs with adjacency matrices $A_G$ and $A_H$. There are equitable partitions $\pi$ in $G$ and $\sigma$ in $H$ with equal symmetrized quotient graphs, that is $\wt{G/\pi} = \wt{H/\sigma}$, if, and only if, there is a nonnegative matrix $M$ satisfying:
    \begin{enumerate}[(i)]
        \item Both $MM^\T$ and $M^\T M$ are doubly stochastic.
        \item $A_GM=MA_H$.
    \end{enumerate}
\end{theorem}

This Theorem is proved in Section~\ref{sec:symquotient}. We also present here in the introduction, in Section~\ref{sec:quantum}, an application of this theorem to the study of quantum walks in graphs. 

In Section~\ref{sec:notionsquotient} we focus on the problem of what can be said about graphs and equitable partitions so that $A_{G/\pi} = A_{H/\sigma}$. We prove that this condition is equivalent to both graphs admitting the same quotient by their coarsest equitable partition. This is done upon examining how the structure of the set of equitable partitions behaves when successive quotients are taken. While it is plausible that the specialist in the topic of equitable partitions will find the results of this section natural or familiar, we have not been able to locate them in the literature, so we state them in this paper as original contributions. We appreciate of course any indication to a suitable reference. We also verify that techniques based on the work in \cite{equitable_partition_and_linearoptimization} provide another equivalent condition to $A_{G/\pi} = A_{H/\sigma}$, along the lines of condition~\eqref{thm1-1} in Theorem~\ref{thm1}, thus we enrich the known Theorem~\ref{thm:leighton} with two extra equivalent characterizations.

Lastly, and on the topic of weighted graphs, we discuss what has been called pseudo-equitable partitions (concept first introduced in \cite{FIOL1996179} and recently studied in \cite{aida_pseudo_eq,aida_pseudo_eq_2} in connection with eigenvalue bounds, though in a sense more restricted than ours). Given $G$, we say that a partition $\pi$ with matrix $P$ is \textsl{pseudo-equitable} if there is a diagonal matrix $D$ of positive vertex weights $u$ so that the colspace of $P$ is $(AD)$-invariant. We denote the quotient graph by $G/(u,\pi)$. This has been particularly useful to generalize the fact that the trivial partition with only one class is equitable if and only if the graph is regular. By putting $D = \Diag(u)$, where $u$ is the Perron-eigenvector of $G$, it follows that the trivial partition with only one class is pseudo-equitable for all graphs. Going further, note now that any two graphs with the same largest eigenvalue will be pseudo-fractionally-isomorphic. More details come in Section~\ref{sec:pseudo}.

The main results in Section~\ref{sec:pseudo} come from analysing the well known results that transform certain matrices into doubly stochastic matrices by diagonal conjugation (see \cite{brualdi1966diagonal,sinkhorn_theorem}) in conjunction to our analysis of equitable partitions. There we show the following theorem.

\begin{theorem}
    Suppose that $G$ and $H$ are positive-weighted graphs with adjacency matrices $A_G$ and $A_H$ respectively. The following are equivalent:
    \begin{enumerate}[(i)]
        \item There is a nonnegative matrix $M$ such that the connected components of the graph $Z_M$ with adjacency matrix
    $$\veto{&M\\M^\T &}$$
    are complete bipartite graphs, and $A_G M=MA_H$;
    \item there are pseudo-equitable partitions $(u,\pi)$ and $(w,\sigma)$ of $G$ and $H$ respectively such that $A_{\wt{G/(u,\pi)}} = A_{\wt{H/(w,\sigma)}}$.
    \item[] If the conditions hold true, the connected components of $M$ are in bijection with pairs of cells of the equitable partitions of each graph.
    \end{enumerate}
\end{theorem}

As a consequence of the previous result, we have a fairly useful description of all pairs of pseudo-equitable partitions from two distinct graphs that admit the same symmetrized quotient.

We point out that if $A_G$ and $A_H$ have the same largest eigenvalue, then both conditions of the previous theorem are always satisfied for the partition with only one part.


\subsection{Equivalent characterizations of fractional isomorphism and other notions of similar equitable partitions} \label{sec:fraciso}

Let $\pi,\sigma$ be two partitions of a graph. We say that $\sigma$ is \textsl{coarser} than $\pi$, and denote by $\pi\leq \sigma$, when each cell of $\pi$ is contained in some cell of $\sigma$. In this case, we can also say that $\pi$ is \textsl{finer} than $\sigma$. Given two partitions, $\pi,\sigma$, there is the finest partition coarser than both, which is said to be $\pi$ \textsl{join} $\sigma$ and is denoted by $\pi\vee \sigma$. Likewise, $\pi\wedge \sigma$ denotes the \textsl{meet}, that is, the coarsest partition finer than both. Equitable partitions form a lattice:

\begin{proposition}[see \cite{Ramana1994}]
    Suppose $\pi,\sigma$ are equitable partitions of $X$. Then $\pi\vee \sigma$ and $\pi\wedge \sigma$ are equitable partitions.
\end{proposition}

Moreover, all graphs have the finest equitable partition, which is the trivial one (each cell is a singleton), and joining all equitable partitions will give us the coarsest equitable partition, which is unique. The equivalence between \eqref{thm1-1} and \eqref{thm1-3} in Theorem~\ref{thm1} was proved in \cite{tinhofer1986graph}. Symbols $D(G)$ and $D(H)$ in item \eqref{thm1-4} in Theorem~\ref{thm1} refer to encodings of all partial outputs of the well-known color refinement algorithm (equivalent to the 1-dimensional Weisfeiler-Leman \cite{weisfeiler1968reduction}), and its equivalence to \eqref{thm1-3} is quite straightforward (see \cite{leighton1982finite}).

Assume $G$ and $H$ are graphs and $\pi$ and $\sigma$ are their coarsest equitable partitions, respectively, so that $G/\pi \simeq H / \sigma$. In \cite[Section 3]{leighton1982finite}, three equivalent conditions to this were presented, all based on the notion of a graph cover. A (possibly infinite) graph $\Gamma$ \textsl{covers} $G$ via maps $\gamma_V : V(\Gamma) \to V(G)$ and $\gamma_E : E(\Gamma) \to E(G)$ if
\begin{itemize}
	\item $u \in e$ in $\Gamma$ implies $\gamma_V(u) \in \gamma_E(e)$ in $G$,
	\item both maps $\gamma_V$ and $\gamma_E$ are onto,
	\item $\gamma_E$ is locally one-to-one, that is, $\gamma_E$ is a bijection between the sets of edges $\{e \in E(\Gamma) : u \in e\}$ and $\{e \in E(G) : \gamma_V(u) \in e\}$.
\end{itemize}
The graph $\Gamma$ is called a cover of $G$. It is straightforward to notice that a connected graph $G$ has a unique tree that covers it, called the universal cover, and unless $G$ is itself a tree, this universal cover is an infinite graph (for instance: if $G$ is a cycle, the universal cover is the infinite path in which all vertices have degree $2$).

\begin{theorem}[see \cite{leighton1982finite}, Section 3] \label{thm:leighton}
	If $G$ and $H$ are graphs, then the following are equivalent:
	\begin{enumerate}[(i)]
		\item $G$ and $H$ share a common finite cover.
		\item $G$ and $H$ have the same universal cover.
		\item $G$ and $H$ share a common (possibly infinite) cover.
		\item If $\pi$ and $\sigma$ are the coarsest equitable partitions of $G$ and $H$ respectively, then $G /\pi \simeq H / \sigma$.
	\end{enumerate}
\end{theorem}

Our work in Section~\ref{sec:notionsquotient} provides an equivalent condition to these:

\begin{theorem}
	Two graphs $G$ and $H$ admit coarsest equitable partitions $\pi$ and $\sigma$ so that $G /\pi \simeq H / \sigma$ if and only if there is a nonnegative nonzero matrix $M$ with constant row sum and constant column sum such that $A_{G} M = M A_{H}$.
\end{theorem}


\subsection{Application to quantum walks} \label{sec:quantum}

A \textsl{continuous-time quantum walk} on a graph $G$ is the map $U$ that takes nonnegative real numbers $t$ and returns a symmetric unitary matrix, as
\[
	U(t) = \exp(\ii t A_G).
\]
There is a very large literature on the topic, and as this is not the main theme of this paper, we will simply refer to the survey \cite{coutinho2018continuous} and references therein. A major problem in this field is to find graphs $G$ so that for some $t$, an off diagonal entry $U(t)$ has absolute value $1$. In this case, $G$ is said to admit \textsl{perfect state transfer}. The connection between this and equitable partitions was first studied in \cite{ge2011perfect}, and we summarize what is of our interest as follows. Say $\pi$ is an equitable partition of $G$ with partition matrix $P$, and denote $A = A_G$ and $B = A_{\wt{G/\pi}}$. Recall that $B$ is symmetric as in \eqref{eq:sympartition}. Then, as $\exp(\ii t A)$ is a polynomial in $A$,
\[
A \wt{P} = \wt{P} B \implies \exp(\ii t A) \wt{P} = \wt{P} \exp(\ii t B),
\]
whence if two classes of $\pi$ are singletons, corresponding to vertices $u$ and $v$, the equation above says that $|\exp(\ii t A)_{uv}| = 1$ if, and only if, $|\exp(\ii t B)_{\{u\}\{v\}}| = 1$. Thus, if any two graphs admit the same symmetrized quotient $B$ with the two vertices as singletons, as above, then perfect state transfer in one implies it in the other. Bottom line: understanding the combinatorics of pairs of graphs that admit the same symmetrized quotient might lead to new ways of constructing examples of perfect state transfer, perhaps allowing for optimizing the size of the graph, as was shown in limited settings in \cite{coutinho2019quantum,kay2018perfect}, and discussed more lengthily in \cite{pereira2023exploring}.


\section{Symmetrized quotient} \label{sec:symquotient}

Our goal in this section is to prove Theorem~\ref{thm:mainthm1}. Again, this has been motivated by the equivalence between \eqref{thm1-1} and \eqref{thm1-2} in Theorem~\ref{thm1}. We start with definitions, and standard lemma (see \cite[Chapter 6]{scheinerman2013fractional}.

A square matrix $M$ is called \textsl{decomposable} if, for some permutation matrices $P$ and $Q$, we have
$$PMQ = \veto{A & 0\\ 0& B},$$
where $A,B$ are square matrices. If $M$ is not decomposable, then it is called \textsl{indecomposable}. It is quite straightforward to realize that for any square matrix $M$, there is a permutation matrix $P$ so that $P\8MP$ is a block matrix, each of its blocks indecomposable.

A square matrix $M$ is called  \textsl{reducible} if, for some permutation matrix $P$, we have
$$PMP\8 = \veto{A & C\\ 0& B},$$
where $A,B$ are square matrices, and $C$ is of appropriate dimension. If $M$ is not reducible, then it is called \textsl{irreducible}, and if $PM$ is irreducible for all permutation matrices $P$, we say that $M$ is \textsl{strongly irreducible}.

\begin{lemma}[see \cite{scheinerman2013fractional}, Proposition 6.2.1] \label{lem:indecomposable}
If $M$ is doubly stochastic and indecomposable, then $M$ is strongly irreducible.
\end{lemma}

This will be useful to us when applied together with the celebrated Perron-Frobenius theorem.

\begin{theorem}[see for instance \cite{horn2012matrix}, Chapter 8]
	Let $M$ be an irreducible, nonnegative matrix. Then among all eigenvalues of $M$ with maximum absolute value, there is one which is positive and its multiplicity is one. This eigenvalue has an associated eigenvector that is entry-wise positive. Further, eigenvectors associated with other eigenvalues are not nonnegative. \label{thm:PF}
\end{theorem}

Recall from Section~\ref{sec:intro} that if $\pi$ is an equitable partition and $S$ denotes its normalized partition matrix, then $\wt{G/\pi}$ denotes the symmetric weighted graph whose adjacency matrix is given by
\[
	A_{\wt{G/\pi}} = S^\T A_G S.
\]

We restate Theorem~\ref{thm:mainthm1} for convenience.

\addtocounter{theorem2}{1}
\begin{theorem2}\label{conditions for equal quotient}
    Let $G$ and $H$ be graphs with adjacency matrices $A= A_G$ and $B = A_H$. There are equitable partitions $\pi$ in $G$ and $\sigma$ in $H$ with equal symmetrized quotient graphs, that is $\wt{G/\pi} \simeq \wt{H/\sigma}$, if, and only if, there is a nonnegative matrix $M$ satisfying:
    \begin{enumerate}[(i)]
        \item Both $MM\8$ and $M\8M$ are doubly stochastic. \label{cond1}
        \item $AM=MB$. \label{cond2}
    \end{enumerate}
\end{theorem2}

\begin{proof}

First we prove that if $G$ and $H$ have equal symmetrized quotients, then conditions \eqref{cond1} and \eqref{cond2} hold. This is quite straightforward.

Let $S$ and $T$ be normalized matrices related to equitable partitions of $G$ and $H$, respectively, satisfying the equation
$$S\8AS=T\8BT.$$
Recalling that $S\8S = T\8T = I$, and that $A$ commutes with $SS\8$ and $B$ commutes with $TT\8$ (as stated right after \eqref{eq:sympartition}), it follows that
\[
A(ST\8)=(ST\8)B	
\]
The matrix $M = ST\8$ is nonnegative, as both $S$ and $T$ are nonnegative, and
$$ST\8(ST\8)\8=S(T\8T)S\8=SS\8,$$
$$(ST\8)\8ST\8=T(S\8S)T\8=TT\8,$$
which are readily seen to be both doubly-stochastic.

Moving on to the converse, suppose a matrix $M$ satisfying conditions \eqref{cond1} and \eqref{cond2} exists. In this case, $MM\8$ commutes with $A$ since
    $$AMM\8=MBM\8=(MBM\8)\8=MM\8A.$$
    
From Lemma \ref{lem:indecomposable}, there exists some permutation $P$ such that
$$P\8 (MM\8) P=\begin{pmatrix}
S_1 &  &  &  &  \\
 & S_2 &  &  &  \\
 &  & S_3 &  &  \\
 &  &  & \ddots &  \\
 &  &  &  & S_k 
\end{pmatrix} ,$$
where $S_r$ is irreducible for all $r$. Let us denote 

$$P\8 A P = \begin{pmatrix}
A_{11} & A_{12} & \dots & A_{1k} \\
A_{21} & A_{22} & \dots & A_{2k} \\
 &  & \ddots &  \\
A_{k1} & A_{k2} & \dots & A_{kk} 
\end{pmatrix},$$
with the same block sizes as $P\8(MM\8)P$. The equation $AMM\8=MM\8A$ is equivalent to
$$S_rA_{rs}=A_{rs}S_s\qquad \forall r,s.$$
Multiplying by $\one$ we get
$$S_rA_{rs}\one=A_{rs}S_s\one,$$
$$S_rA_{rs}\one=A_{rs}\one,$$
and because $S_r$ is irreducible, by the Perron-Frobenius Theorem we have that $A_{rs}\one=c\one$, which means that these blocks correspond to a partition $\pi$ of $V(G)$ which is equitable (this previous argument is essentially equivalent to the implication \eqref{thm1-1} to \eqref{thm1-2} from Theorem~\ref{thm1}, and we reference \cite[Section 6]{scheinerman2013fractional} for a detailed exposition).

Let us denote $\pi = \{C_1,\dots,C_k\}$, each $C_i \subseteq V(G)$. By an analogous argument, the partition of $M\8M$ into indecomposable blocks induces an equitable partition $\sigma$ of $H$, say $\sigma = \{D_1,\dots,D_m\}$, $D_i \subseteq V(H)$.

In $\pi$, we will have that $a$ and $b$ are in the same cell if, and only if, they are in some irreducible block of $MM\8$, that is, $(MM\8)^\ell_{a,b}\neq 0$ for some $\ell$. Let us say this cell is $C_r$, and denote by $\one_r$ its indicator vector. We will show that $M\8\one_r$ corresponds in some way to a cell of $\sigma$. 

\begin{itemize}
	\item $M\8\one_r$ is supported in one cell of $\sigma$.
	\item[] We suppose that $a$ and $b$ are in the support of $M\8\one_r$ and prove that $(M\8M)^\ell_{a,b}\neq 0$ for some $\ell$. The support is non-empty, for $MM\8\one_r\neq 0$. Using its decomposition in blocks, each $1$-eigenvector $v$ satisfies
$$P\begin{pmatrix}
S_1 &  &  &  &  \\
 & S_2 &  &  &  \\
 &  & S_3 &  &  \\
 &  &  & \ddots &  \\
 &  &  &  & S_k 
\end{pmatrix} P\8v=v,$$
from where each block of $P\8 v$ is an eigenvector of $S_s$. As $S_s$ is indecomposable (and doubly-stochastic), by Perron-Frobenius, we must have that $v$ is constant in these blocks. Hence, $\{\one_s\}_s$ forms a basis for the $1$-eigenspace of $MM\8$ and the related eigenprojector is $\sum \alpha_s \one_s\one_s\8$, for $\alpha_s=\nm{\one_s}^{-2}>0$.

Let $f(x)=\sum c_\ell x^\ell$ be such that $f(MM\8)=\sum\alpha_s\one_s\one_s\8$. If $a$ and $b$ are not in the same indecomposable block of $M\8M$, then $((M\8M)^\ell)_{a,b}=0$ for any $\ell>0$, whence
\begin{align*}
    0&=\prt{\sum c_\ell (M\8M)^{\ell+1}}_{a,b}\\
    &=\prt{M\8\prt{\sum c_\ell (MM\8)^\ell}M}_{a,b}\\
    &=\prt{M\8\prt{\sum \alpha_s\one_s\one_s\8} M}_{a,b}\\
    &\geq \alpha_r\prt{M\8\one_r\one_r\8M}_{a,b}\\
    &=\alpha_r(M\8\one_r)_a(M\8\one_r)_b\\
    &>0,
\end{align*}
a contradiction. Hence, we may conclude that both $a$ and $b$ belong to the same cell within the equitable partition of $B$ associated with $M\8M$.

	\item $M\8\one_r$ is constant on each cell of $\sigma$.
	\item[] Let $\eta : \{1,\dots,k\} \to \{1,\dots,m\}$ defined by having $D_{\eta(r)}$ be the cell of $\sigma$ that contains the support of $M\8\one_r$. We denote the indicator vector of $D_{\eta(r)}$, which is in $\R^{V(H)}$, by $\one_{\eta(r)}$.
	
	Considering the expression $(M\8M)(M\8\one_r) = M\8(MM\8\one_r) = M\8\one_r$, we observe that $M\8\one_r$ is an eigenvector of $M\8M$, just like $\one_{\eta(r)}$, and both are nonzero only on the entries corresponding to the submatrix of $M\8M$ with entries in $D_{\eta(r)}$, which is irreducible. Therefore $M\8\one_r$ is parallel to $\one_{\eta(r)}$ by the Perron-Frobenius theorem.
\end{itemize}

We shall now demonstrate that the function $\eta$ establishes an isomorphism between $\wt{G/\pi}$ and $\wt{H/\sigma}$.

Since $\sum\one_r=\one$ and $M\8\one$ is supported on all entries (if $M\8$ has a row equal to zero, then $M\8M$ would not be doubly-stochastic), and since $M\8$ is nonnegative, it follows that all cells of the equitable partitions are covered by $\{\one_{\eta(r)}\}_r$, or, in others words, $\eta$ is onto.

Also, $\eta$ is one-to-one, since:
$$\one_{\eta(r)}=\one_{\eta(s)}\implies M\8\one_r\parallel M\8\one_s\implies MM\8\one_r\parallel MM\8\one_s\implies r=s.$$

In particular $k = m$.

At last, we shall prove that $\eta$ is a graph isomorphism. We begin by noting that
$$\nm{M\8\one_r}^2=\one_r\8 (MM\8\one_r)=\one_r\8\one_r,$$
from where
$$\frac{M\8\one_r}{\sqrt{\one_r\8\one_r}}=\frac{\one_{\eta(r)}}{\sqrt{\one_{\eta(r)}\8\one_{\eta(r)}}}.$$

Thus, a direct calculation gives us

\begin{align*}
    \frac{\one_{\eta(r)}\8}{\sqrt{\one_{\eta(r)}\8\one_{\eta(r)}}}B\frac{\one_{\eta(s)}}{\sqrt{\one_{\eta(s)}\8\one_{\eta(s)}}}&=\frac{\one_r\8}{\sqrt{\one_r\8\one_r}}MBM\8\frac{\one_s}{\sqrt{\one_s\8\one_s}}\\
    &=\frac{\one_r\8}{\sqrt{\one_r\8\one_r}}AMM\8\frac{\one_s}{\sqrt{\one_s\8\one_s}}\\
    &=\frac{\one_r\8}{\sqrt{\one_r\8\one_r}}A\frac{\one_s}{\sqrt{\one_s\8\one_s}}
\end{align*}
or, likewise
\[
	(A_{\wt{G/\pi}})_{rs} = (A_{\wt{H/\sigma}})_{\eta(r)\eta(s)}.
\]
which completes the proof that $\eta$ is an isomorphism.
\end{proof}

\begin{corollary}
   Let $G$ and $H$ be graphs with adjacency matrices $A= A_G$ and $B = A_H$. Assume there is a nonnegative matrix $M$ satisfying:
    \begin{enumerate}[(i)]
        \item Both $MM\8$ and $M\8M$ are doubly stochastic. 
        \item $AM=MB$.
    \end{enumerate}
	Consider the graph $Z_M$ whose adjacency matrix is the support of
    $$\veto{& M\\ M\8&}.$$
    Let $\pi = \{C_1,\dots,C_k\}$ and $\sigma=\{D_1,\dots,D_k\}$ be the equitable partitions as in Theorem~\ref{thm:mainthm1}. Then $Z_M$ is a bipartite graph with $k$ connected components, each equal to $C_r \cup D_r$, for $r \in \{1,\dots,k\}$.
\end{corollary}

\begin{proof}
	We observe that two vertices $a,b \in V(G)$ are in the same cell if and only if there is some $\ell$ for which $(MM\8)^\ell_{a,b}\neq 0$. This is equivalent to having some (even) path between $a$ and $b$ in the graph $Z_M$. Likewise for if these are vertices in $V(H)$.
	
	To see that vertices in $C_r$ are connected to vertices in $D_r$, let $\one_r$ be the indicator vector of $C_r$ in $\R^{V(G)}$. Note that $M\8\one_r$ is parallel to the indicator vector of $D_r$, both nonzero, thus some entry $M^\T$ whose row corresponds to a vertex of $D_r$ and column corresponds to a vertex of $C_r$ is nonzero. In other words, $C_r \cup D_r$ belong to the same connected component.
\end{proof}

\section{Different notions of common quotients} \label{sec:notionsquotient}

Theorems \ref{thm1},\ref{thm:leighton} and \ref{thm:mainthm1} describe equivalent conditions to three relations between graphs $G$ and $H$, namely
\begin{enumerate}[(A)]
	\item There are equitable partitions $\pi$ and $\sigma$ of $G$ and $H$ respectively, with partition matrices $P$ and $S$, so that \label{conditionA}
	\[P\8P = S\8S \quad \text{and} \quad P\8A_GP = S\8 A_H S.\]
	(This is the condition used in Theorem \ref{thm1}.)
	\item There are equitable partitions $\pi$ and $\sigma$ of $G$ and $H$ respectively, with partition matrices $P$ and $S$, so that \label{conditionB} \[(P\8P)^{-1} P\8A_GP = (S\8S)^{-1} S\8 A_H S.\]
	(This is the condition used in Theorem \ref{thm:leighton}.)
	\item There are equitable partitions $\pi$ and $\sigma$ of $G$ and $H$ respectively, with normalized partition matrices $\wt P$ and $\wt S$, so that \label{conditionC} \[\wt{P\8}A_G\wt{P} = \wt{S\8} A_H \wt{S}.\]
	(This is the condition used in Theorem \ref{thm:mainthm1}.)
\end{enumerate}

The attentive reader will have noticed that while Theorem~\ref{thm:leighton} speaks about the coarsest equitable partitions, we have stated relation \eqref{conditionB} with no such restriction. This is fine because of the following result, which should not be surprising but we could not find a reference to it.

Recall the lattice of equitable partitions introduced in Section~\ref{sec:fraciso}, which we denote by $\Pi(G)$. If $\sigma \in \Pi(G)$, we denote its characteristic matrix by $P_\sigma$, and $D_\sigma = (P_\sigma\8 P_\sigma)^{-1}.$ Because a partition corresponds to a partition matrix, we might abuse the notation and say that $P \in \Pi(G)$ if $P = P_\pi$, when $\pi \in \Pi(G)$.

\begin{proposition} \label{prop:eqpartitionquotient}
    Let $G$ be a graph on $n$ vertices. Let $\sigma\in\Pi(G)$. Define the function
    \begin{align*}
    	\varphi\colon\{\pi\in\Pi(G)\colon\sigma\leq \pi\} &\mt \Pi(G/\sigma), \\
    	\pi & \mapsto D_\sigma P_{\sigma}\8P_{\pi}.
    \end{align*}
    Then $\varphi$ is an order-preserving bijection. If $S \in \Pi(G/\sigma)$, its inverse is given by
    \[
    	\varphi^{-1}(S) = P_{\sigma} S,
    \]
    and we also have
    $$(G / \sigma)/\varphi(\pi)\simeq G/\pi.$$
    
\end{proposition}

\begin{proof}
    Fix some $\pi\in\Pi(G)$. The fact that $D_\sigma P_{\sigma}\8P_{\pi}$ is a partition matrix of $V(G / \sigma)$ follows immediately --- it is the partition induced by $\pi$ on the classes of $\sigma$. To see that it is an equitable partition, we proceed as follows:
    \begin{align}
    	A_{G / \sigma} (D_\sigma P_{\sigma}\8P_{\pi}) & = (D_\sigma P_{\sigma}\8P_{\sigma}) A_{G / \sigma} (D_\sigma P_{\sigma}\8P_{\pi}) \label{eq3} \\
    	& = (D_\sigma P_{\sigma}\8) A P_{\sigma} (D_\sigma P_{\sigma}\8P_{\pi}) \label{eq4}  \\
    	& = (D_\sigma P_{\sigma}\8) A P_{\pi} \label{eq5} \\
    	& = (D_\sigma P_{\sigma}\8 P_{\pi}) A_{G / \pi}\label{eq6}
    \end{align}
    where \eqref{eq3} follows from $D_\sigma P_{\sigma}\8P_{\sigma} = I$; \eqref{eq4} from the fact that $\sigma$ is an equitable partition, thus $A P_{\sigma} = P_{\sigma} A_{G / \sigma}$; \eqref{eq5} because $P_{\sigma} D_\sigma P_{\sigma}\8$ is an orthogonal projection and the columns of $P_{\pi}$ are in its range; and \eqref{eq6} because $\pi$ is an equitable partition of $G$.
    
    To see that $\varphi$ is one-to-one, a simple calculation gives us
    $$D_\sigma P_\sigma\8 P_\pi=D_\sigma P_\sigma\8 P_\nu \implies P_\sigma D_\sigma P_\sigma\8 P_\pi =P_\sigma D_\sigma P_\sigma\8 P_\nu  \implies P_\pi =P_\nu.$$

    For each $S$ that is a partition matrix of an equitable partition of $G/\sigma$, we have $A_{G /\sigma} S = S C$ for some $C$. Applying $P_\sigma$ we get
    $$P_\sigma A_{G /\sigma} S = P_\sigma S C \implies A_G P_\sigma S = P_\sigma S C.$$
    As each line of $P_{\sigma}$ has only one nonzero entry, we can also conclude that $P_{\sigma} S$ is a matrix of $0$'s and $1$'s. Thus, it is an equitable partition of $G$. As $D_mu P_\sigma\8 P_\sigma S = S$, the image of $P_\sigma S$ is indeed $S$, as we wanted.

    Finally, we note that $\pi\leq\nu$ if and only if $\rng P_\pi \supseteq \rng P_\nu$, and as the composition of functions preserves the inclusion of their range, we conclude that $\varphi$ preserves order.
\end{proof}

As a consequence, it follows that the quotient graph by the coarsest equitable partition of $G$ coincides with the quotient graph by the coarsest equitable partition of any of its quotients graphs by other equitable partitions. We may therefore add one extra equivalent condition to the formulation of Theorem~\ref{thm:leighton}, justifying the phrasing of Condition~\eqref{conditionB}. Motivated partly by the work in \cite{equitable_partition_and_linearoptimization}, we provide a third characterization after the lemma.

\begin{lemma} \label{lemma11}
    Let $G,H$ be graphs with equitable partitions $\pi$ and $\sigma$ and $\varphi\colon C_r\mapsto D_r$ a map between the cells of $\pi$ and $\sigma$. Then $\varphi$ induces an isomorphism from $G/\pi$ to $H/\sigma$ if and only if $\varphi$ induces an isomorphism from $\widetilde{G/\pi}$ to $\widetilde{H/\sigma}$ and there is some constant $\lambda$ for which $|C_r|=\lambda|D_r|$.
\end{lemma}
\begin{proof}
	First off, recall that if $G/\pi$ and $H/\sigma$ are isomorphic, then we may assume $A_{G / \pi} = A_{H / \sigma}$, and because there is only one way to symmetrize these matrices by diagonal similarity, it follows that $A_{\wt{G / \pi}} = A_{\wt{H / \sigma}}$, so it follows that $\widetilde{G/\pi}$ and $\widetilde{H/\sigma}$ are isomorphic.

    We observe that $\widetilde{G/\pi}\simeq\widetilde{H/\sigma}$ is equivalent to having a convenient indexing which satisfies
    $$\frac{N(C_r,C_s)}{\sqrt{|C_r||C_s|}}=\frac{N(D_r,D_s)}{\sqrt{|D_r||D_s|}}\qquad\forall r,s;$$
    which is also equivalent to
    $$\frac{N(C_r,C_s)}{|C_r|}=\prt{\frac{\sqrt{|D_r|}}{\sqrt{|C_r|}}\cdot\frac{\sqrt{|C_s|}}{\sqrt{|D_s|}}}\frac{N(D_r,D_s)}{|D_r|}\qquad\forall r,s.$$
    Recalling that
    $$(A_{G/\pi})_{r,s}=\frac{N(C_r,C_s)}{|C_r|},$$
    we have an isomorphism in the non-symmetrized quotient if and only if
    $$\frac{\sqrt{|D_r|}}{\sqrt{|C_r|}}\cdot\frac{\sqrt{|C_s|}}{\sqrt{|D_s|}}=1\qquad\forall r,s;$$
    which is equivalent to the desired condition.
\end{proof}

Now, we can enunciate and prove the following:

\begin{theorem} \label{thm:combinatorialquotient}
	If $G$ and $H$ are graphs, then the following are equivalent:
	\begin{enumerate}[(i)]
		\item If $\pi$ and $\sigma$ are the coarsest equitable partitions of $G$ and $H$ respectively, then $G / \pi \simeq H /\sigma$. \label{comq1}
		\item There are equitable partitions $\pi$ and $\sigma$ of $G$ and $H$ respectively such that $G / \pi \simeq H /\sigma$. \label{comq2}
		\item There is a nonnegative and nonzero matrix $M$ with constant row sum and constant column sum such that $A_{G} M = M A_{H}$. \label{comq3}
	\end{enumerate}
\end{theorem}
\begin{proof}
	As shown in Proposition~\ref{prop:eqpartitionquotient}, the first two conditions are equivalent. 
 
    We start by proving \eqref{comq3} $\implies$ \eqref{comq2}. Normalizing $M$, we may assume that $MM\8$ and $M\8M$ are doubly stochastic. Applying Theorem \ref{thm:mainthm1}, it follows that $\widetilde{G/\pi}\simeq\widetilde{H/\sigma}$, for the equitable partitions $\pi, \sigma$ induced by $M$.

    Let $\one_{C_r}$ be the indicator of $C_r$. By the argument in the proof of Theorem \ref{thm:mainthm1}, we have $M\one_{C_r}$ parallel to the indicator vector of the cell isomorphic to $C_r$ in the partition $\sigma$ of $H$, which we will denote by $D_r$. As $M\one$ is a multiple of $\one$, and the vectors $M\one_{C_r}$ have disjoint support, then for each $r$, $M\one_{C_r}$ must be a multiple of $\one_{D_r}$ by a factor that is independent of $r$. Hence, using the previous lemma we obtain the result.

    For the converse, recall that $A_G\wt{P} \wt{S}\8 = \wt{P} \wt{S}\8A_H$, where $\wt{P},\wt{S}$ are the normalized matrices of $\pi,\sigma$. A simple calculation gives us that
    \begin{align*}
        (\wt{P} \wt{S}\8 \one)_k & = \sum_{r,s} \wt{P}_{k,r}\wt{S}_{r,s}\\
        &=\frac{1}{\sqrt{|C_r|}}\sum_{s\in D_r}\frac{1}{\sqrt{|D_r|}},\quad\text{where $k\in C_r$,}\\
        &=\sqrt{\frac{|D_r|}{|C_r|}}\\
        &=\gamma.
    \end{align*}
    The existence of a value of $\gamma$ independent of $r$ follows from the lemma above. An analogous argument shows that $\wt{P} \wt{S}\8$ is column stochastic.
\end{proof}

Coming back to the conditions in the beginning of this section, note the obvious implication $\eqref{conditionA} \implies \eqref{conditionB} \implies \eqref{conditionC}$. None of them is an equivalence, and we show below a large class of examples satisfying only condition \eqref{conditionC}.
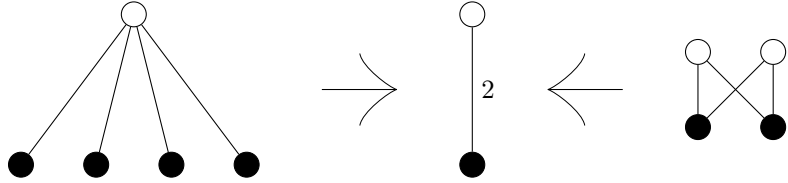
      $$\begin{tikzpicture}
    \draw (1.5,2) node[draw,circle](q1){};
    \draw (0,0) node[draw,circle,fill=black](q2){};
    \draw (1,0) node[draw,circle,fill=black](q3){};
    \draw (2,0) node[draw,circle,fill=black](q4){};
    \draw (3,0) node[draw,circle,fill=black](q5){};
    
    \draw[-] (q1) edge node{} (q2);
    \draw[-] (q1) edge node{} (q3);
    \draw[-] (q1) edge node{} (q4);
    \draw[-] (q1) edge node{} (q5);

    \draw (6,0) node[draw,circle,fill=black](r1){};
    \draw (6,2) node[draw,circle](r2){};
    
    \draw[-] (r1) edge[right] node{$2$} (r2);
    
    \draw [-{Classical TikZ Rightarrow[length=5mm]}] (4,1) -- (5,1);
    \draw [-{Classical TikZ Rightarrow[length=5mm]}] (8,1) -- (7,1);

    \draw (9,1.5) node[draw,circle](1){};
    \draw (10,1.5) node[draw,circle](2){};
    \draw (9,0.5) node[draw,circle,fill=black](3){};
    \draw (10,0.5) node[draw,circle,fill=black](4){};
    
    \draw[-] (1) edge node{} (3);
    \draw[-] (1) edge node{} (4);
    \draw[-] (2) edge node{} (3);
    \draw[-] (2) edge node{} (4);
    
\end{tikzpicture}$$
\noindent
In general, it is easy to see that $K_{r,s}$ and $K_{r',s'}$ have a common symmetrized quotient if $rs=r's'$, and naturally if $r \neq r'$ it cannot be that they admit the same combinatorial quotient.

While Conditions \eqref{conditionA} and \eqref{conditionB} are clearly equivalence relations, Condition \eqref{conditionC} is not, as the following example shows:

$$
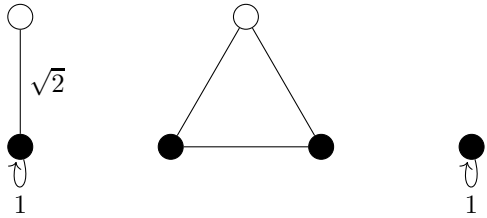
\begin{tikzpicture}
    \draw (6,0) node[draw,circle,fill](qq1){};
    \draw[-] (qq1) edge[loop below, below] node{1} (qq1);

    \draw (3,1.73) node[draw,circle](q1){};
    \draw (2,0) node[draw,circle,fill](q2){};
    \draw (4,0) node[draw,circle,fill](q3){};
    \draw[-] (q1) edge node{} (q2);
    \draw[-] (q2) edge node{} (q3);
    \draw[-] (q3) edge node{} (q1);

    \draw (0,1.73) node[draw,circle](1){};
    \draw (0,0) node[draw,circle,fill](2){};
    \draw[-] (1) edge[right] node{$\sqrt{2}$} (2);
    \draw[-] (2) edge[loop below, below] node{1} (2);
\end{tikzpicture}$$
In which the first graph is related to the second, and the second to the third but the first and last one are not related.


\begin{lemma}[see for instance Theorem 6.2.4 in \cite{scheinerman2013fractional}]
	If $S$ and $T$ are doubly stochastic matrices and $v = Su$ and $u = Tv$, then $v = Pu$, for some permutation matrix $P$.
\end{lemma}

\section{Pseudo-equitable partitions --- a linear algebra characterization}

The concept of a pseudo-equitable partition was likely introduced in \cite{FIOL1996179}, and has been recently studied in \cite{aida_pseudo_eq} (see references therein for other works on the topic). Essentially a partition is pseudo-equitable if there is a positive vertex weighting that makes it equitable. In most cases, it turns out that this vertex weighting is given by the eigenvector corresponding to the largest eigenvalue (the Perron-eigenvector, as it is sometimes called), but in this section we allow for any positive weighting.

Let $w\in \R^V_+$ be a vector with all positive entries, and let $D_w = \Diag(w)$ be the diagonal matrix whose diagonal entries correspond to $w$. We say that a partition $\pi$ of $V(G)$ is pseudo-equitable with respect to $w$ if the column space of $D_wP_\pi$ is $A_G$-invariant, or better yet, if there is a matrix $B$ such that
\[
	A_G (D_w P_\pi) = (D_w P_\pi) B.
\]
It is immediate to note that an equivalent way of seeing this is to consider that $\pi$ is equitable for the weighted graph $D_w\1 A_G D_w$. We say that the matrix $B$ is the adjacency matrix of the quotient graph $G/(w,\pi)$, that is, $B = A_{G / (w,\pi)}$.

We are particularly invested in scenarios where two graphs admit the same quotient graph, but perhaps one of them by means of a pseudo-equitable partition.

\begin{figure}[h]
\[	\begin{tikzpicture}
	\draw (1,1) node[draw,circle](q1){};
    \draw (2,1.9) node[draw,circle](q2){};
    \draw (2,1.3) node[draw,circle](q3){};
    \draw (2,0.7) node[draw,circle](q4){};
    \draw (2,0.1) node[draw,circle](q5){};
    \draw (3,2) node[draw,circle](a1){};
    \draw (3,1.6) node[draw,circle](a2){};
    \draw (3,1.2) node[draw,circle](a3){};
    \draw (3,0.8) node[draw,circle](a4){};
    \draw (3,0.4) node[draw,circle](a5){};
    \draw (3,0) node[draw,circle](a6){};
    \draw (4,1.9) node[draw,circle](q6){};
    \draw (4,1.3) node[draw,circle](q7){};
    \draw (4,0.7) node[draw,circle](q8){};
    \draw (4,0.1) node[draw,circle](q9){};
    \draw (5,1) node[draw,circle](q10){};
    \draw[-] (q1) edge node{} (q2);
    \draw[-] (q1) edge node{} (q3);
    \draw[-] (q1) edge node{} (q4);
    \draw[-] (q1) edge node{} (q5);
    \draw[-] (q2) edge node{} (a1);
    \draw[-] (q2) edge node{} (a2);
    \draw[-] (q2) edge node{} (a3);
    \draw[-] (q3) edge node{} (a1);
    \draw[-] (q3) edge node{} (a4);
    \draw[-] (q3) edge node{} (a5);
    \draw[-] (q4) edge node{} (a2);
    \draw[-] (q4) edge node{} (a4);
    \draw[-] (q4) edge node{} (a6);
    \draw[-] (q5) edge node{} (a3);
    \draw[-] (q5) edge node{} (a5);
    \draw[-] (q5) edge node{} (a6);
	\draw[-] (q6) edge node{} (a1);
    \draw[-] (q6) edge node{} (a2);
    \draw[-] (q6) edge node{} (a3);
    \draw[-] (q7) edge node{} (a1);
    \draw[-] (q7) edge node{} (a4);
    \draw[-] (q7) edge node{} (a5);
    \draw[-] (q8) edge node{} (a2);
    \draw[-] (q8) edge node{} (a4);
    \draw[-] (q8) edge node{} (a6);
    \draw[-] (q9) edge node{} (a3);
    \draw[-] (q9) edge node{} (a5);
    \draw[-] (q9) edge node{} (a6);
    \draw[-] (q10) edge node{} (q6);
    \draw[-] (q10) edge node{} (q7);
    \draw[-] (q10) edge node{} (q8);
    \draw[-] (q10) edge node{} (q9);
\end{tikzpicture} \hspace{1cm} 
\begin{tikzpicture}
	\draw (1,1) node[draw,circle](q1){};
    \draw (2,1.9) node[draw,circle](q2){};
    \draw (2,1.3) node[draw,circle](q3){};
    \draw (2,0.7) node[draw,circle](q4){};
    \draw (2,0.1) node[draw,circle](q5){};
    \draw (3,2) node[draw,circle](a1){};
    \draw (3,1) node[draw,circle,fill=black](a11){};
    \draw (3,0) node[draw,circle](a6){};
    \draw (4,1.9) node[draw,circle](q6){};
    \draw (4,1.3) node[draw,circle](q7){};
    \draw (4,0.7) node[draw,circle](q8){};
    \draw (4,0.1) node[draw,circle](q9){};
    \draw (5,1) node[draw,circle](q10){};
    \draw[-] (q1) edge node{} (q2);
    \draw[-] (q1) edge node{} (q3);
    \draw[-] (q1) edge node{} (q4);
    \draw[-] (q1) edge node{} (q5);
    \draw[-] (q2) edge node{} (a1);
    \draw[-] (q2) edge node{} (a11);
    \draw[-] (q3) edge node{} (a1);
    \draw[-] (q3) edge node{} (a11);
    \draw[-] (q4) edge node{} (a11);
    \draw[-] (q4) edge node{} (a6);
    \draw[-] (q5) edge node{} (a11);
    \draw[-] (q5) edge node{} (a6);
	\draw[-] (q6) edge node{} (a1);
    \draw[-] (q6) edge node{} (a11);
    \draw[-] (q7) edge node{} (a1);
    \draw[-] (q7) edge node{} (a11);
    \draw[-] (q8) edge node{} (a11);
    \draw[-] (q8) edge node{} (a6);
    \draw[-] (q9) edge node{} (a11);
    \draw[-] (q9) edge node{} (a6);
    \draw[-] (q10) edge node{} (q6);
    \draw[-] (q10) edge node{} (q7);
    \draw[-] (q10) edge node{} (q8);
    \draw[-] (q10) edge node{} (q9);
\end{tikzpicture}\]
\caption{Example of two graphs that admit the same quotient graph with respect to the classes with the same horizontal coordinate. However for the graph on the right hand side it is necessary to use a pseudo-equitable partition whose vector of weights assigns 1 to all white vertices, and 2 to the black vertex.}  \label{pstquot}
\end{figure}
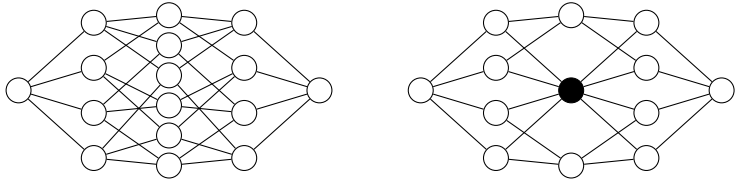

One application comes from quantum walks, as discussed in Section~\ref{sec:fraciso}. The graph on the left hand side in Figure~\ref{pstquot} admits perfect state transfer because it is a cartesian power of $P_2$, and therefore its symmetrized quotient graph also does. Because the graph on the right hand side admits the same quotient, it also admits perfect state transfer (see \cite[Section 5]{coutinho2016spectrally}).

In the next section we will characterize when two graphs admit the same symmetrized quotient with respect to pseudo-equitable partitions, but let us first show a theorem that associates to pseudo-equitable partitions very natural objects.

We need the following lemma, whose proof can be found in \cite{brosch2022jordan}.

\begin{lemma}\label{lemma from somewhere}
    An orthogonal projection matrix $P$ has nonnegative entries if and only if $\rng P$ has an orthonormal basis with only nonnegative vectors.
\end{lemma}

For the next theorem, it will be convenient to assume all pseudo-equitable partitions are presented with weight vectors that have been normalized in each class of the partition (this is without loss of any generality). We denote by $\hat\Pi(G)$ the set of all pairs of the form $(w,\pi)$ where $\pi$ is a pseudo-equitable partitions of $G$ with weight vector $w$, and $w$ is normalized in each class of $\pi$.

Let $\mathcal{P}(G)$ denote the set of all orthogonal projection matrices with nonnegative entries that satisfy the following two properties
\begin{enumerate}[(A)]
	\item If $M \in \mathcal{P}(G)$, then there is a positive vector $u$ in its range.
	\item $M$ commutes with $A_G$. 
\end{enumerate}
Note that if $M \in \mathcal{P}(G)$, we may decompose it into irreducible blocks. Each block has rank 1 and contains a unique positive eigenvector to the eigenvalue 1, as a consequence of the Perron-Frobenius theorem (Theorem \ref{thm:PF}). The vector which is the sum of the normalized Perron eigenvectors of each irreducible block is positive and belongs to the range of $M$. We call it the \textit{special positive eigenvector} of $M$.

Given $(w,\pi) \in \hat\Pi(G)$, we let $P_\pi$ denote its (unweighted) characteristic matrix, and $P_{(w,\pi)} = D_w P_\pi$ its weighted characteristic matrix. Note that its columns are already normalized. Then we denote by $S_{(w,\pi)}$ the projection onto the column space of $P_{(w,\pi)}$, that is
\begin{equation}
	S_{(w,\pi)} = P_{(w,\pi)}P_{(w,\pi)}\8 = D_wP_\pi P_\pi\8 D_w. \label{eq:wpi}
\end{equation}
With that, we can relate projectors and pseudo-equitable partitions by the following theorem.

\begin{theorem}
    Let $G$ be a graph. The following function is a bijection
    \begin{align*}
    	\varphi\colon\hat\Pi(G) &\mt \mathcal{P}(G),\\
    	(w,\pi) & \mapsto S_{(w,\pi)}.
    \end{align*}
     Then, if $\pi$ and $\sigma$ are pseudo-equitable partitions with respect to the same weight vector $w$, now dropped from the subscripts for clarity, it follows that:
    \begin{enumerate}[(i)]
        \item $\pi\leq\sigma\iff S_\pi\succeq S_\sigma\iff \rng S_\pi\supseteq  \rng S_\sigma,$ \label{condi}
        \item $S_{\pi\vee\sigma}$ is the orthogonal projection onto $\rng S_\pi \cap \rng S_\sigma$. \label{condii}
    \end{enumerate}
\end{theorem}
\begin{proof}
    Given $(w,\pi)$, the matrix $S_{(w,\pi)}$ is a nonnegative orthogonal projection by definition, and it commutes with $A_G$ due to an argument analogous to \cite[Lemma 1.1]{Compact_graphs_and_equitable_partitions} and shown in \cite[Lemma 3.1]{aida_pseudo_eq_2} for when $w$ is the Perron-eigenvector. Equation~\eqref{eq:wpi} gives an immediate way to verify that $w$ is an eigenvector of $S_{(w,\pi)}$, and in fact it is going to be the special positive eigenvector of $S_{(w,\pi)}$.
    
    Given two elements of $\hat\Pi(G)$, if they correspond to different partitions, then clearly their images under $\varphi$ are different. If the partitions are equal but the weight vectors are different, the unicity of the special positive eigenvector implies that their images will be different. Thus $\varphi$ is injective.

	Now, let $P\in \mathcal{P}(G)$ be arbitrary. Using Lemma \ref{lemma from somewhere}, we have an orthonormal basis $\zeta_r$ of nonnegative vectors for $\rng P$. As $\zeta_r$ are nonnegative, we must have for $r\neq s$ that $\supp\zeta_r\cap\supp\zeta_s=\emptyset$ to ensure orthogonality, and because there is at least one positive eigenvector of $P$, the supports of the $\zeta_r$ correspond to a decomposition of $P$ into irreducible blocks. The sum of the $\zeta_r$ is the special positive eigenvector of $P$, which we call $u$. Thus we conclude that $P = S_{(u,\pi)}$, where $\pi$ is the partition into $\{\supp\zeta_r\}_r$. The fact that $\pi$ is pseudo-equitable with weighting $u$ follows from the fact that $P$ and $A$ commute (again, see \cite[Lemma 1.1]{Compact_graphs_and_equitable_partitions} or \cite[Lemma 3.1]{aida_pseudo_eq_2}).
	
	Fix $w$ and assume $(w,\pi)$ and $(w,\sigma)$ are pseudo-equitable. We drop $w$ from the subscripts for clarity.
	
	Condition~\eqref{condi} follows from the fact that $\pi\leq\sigma$ if and only if $\col P_\pi \supseteq\col P_\sigma$, combined with the property $\rng S_\pi=\col P_\pi$, and that $S_\pi$ and $S_\sigma$ are projections.

	Based on \eqref{condi}, we can conclude that $\rng S_{\pi\vee\sigma}$ is the largest vector space which is contained in both $\rng S_\pi$ and $\rng S_\sigma$, and whose projector is also in $\mathcal{P}(G)$. The projector onto $\rng S_\pi\cap\rng S_\sigma$ is indeed in $\mathcal{P}(G)$ since it is the limit of nonnegative operators that commute with $A$: let $S=\lim_n (S_\pi S_\sigma S_\pi)^n$ --- it indeed exists and is a projector, as all entries of the Jordan matrix of $S_\pi S_\sigma S_\pi$ converge to zero, except the ones in the diagonal. It is self-adjoint as it is the limit of self-adjoint operators, and $\rng S=\rng S_\pi\cap \rng S_\sigma$ because as $S_\pi S_\sigma S_\pi$ is a contraction, the fix points of the convergent are the same as the fixed points of $S_\pi S_\sigma S_\pi$, which are $\rng S_\pi\cap \rng S_\sigma$.

\end{proof}


\section{Characterizing pseudo-equitable partitions} \label{sec:pseudo}

In this section, we present an analysis of when it is possible to use the technology from Theorem \ref{conditions for equal quotient} to understand pseudo-equitable partitions. 

Given graphs $G$ and $H$, suppose that for some nonzero nonnegative matrix $M$ we have $A_G M = M A_H$. Is there a weighting of these graphs that result in pseudo-equitable partitions whose quotients are the same? It turns out that this is the case provided there are diagonal matrices $D,E$ such that $DME^2M\8D$ and $EM\8D^2ME$ are doubly stochastic. Are these always available?

We will use a well-known theorem (see for instance \cite{sinkhorn_theorem}) in our result. Beforehand, a definition. We say that an $n\times n$ matrix $M$ has {\sl total support} if for each nonzero entry $M_{a,b}$ we have some permutation $\sigma\in S_n$ such that $\sigma a=b$ and for each $r$, $M_{r,\sigma r}\neq 0$.

\begin{theorem}[from \cite{sinkhorn_theorem}] \label{thm:sink}
    Given a nonnegative square matrix $M$, there are diagonal matrices $D, E$ with positive nonzero entries such that $DME$ is doubly stochastic if and only if $M$ has total support. Moreover, the matrix $DME$ is unique; and $D,E$ are also unique up to scalar multiplication if and only if $M$ is fully indecomposable.
\end{theorem}

The following result can be found in \cite{symetric_sinkhorn}.
\begin{corollary} \label{cor:sinkcor}
    Given a symmetric matrix $M$ with total support, we have a unique diagonal matrix $D$ with positive nonzero entries such that $DMD$ is doubly stochastic.
\end{corollary}

We now start by describing which conditions on $M$ we need in order to apply the theorem. We will from now on suppose that $M$ is $k\times l$ and call $R_r$ the $r$th row-vector of $M$ and $C_r$ the $r$th column-vector.

\begin{lemma} \label{lem:17}
    Given a matrix $M$ we have that $MM\8$ has total support if and only if $M$ has no null rows. Moreover, in this case, $MDM\8$ also has total support for any positive diagonal matrix $D$.
\end{lemma}
\begin{proof}
    We start by observing that 
    $$(MM\8)_{a,b}\neq 0\iff R_a\not\perp R_b.$$
    If $MM\8$ has total support, then it is immediate from the definition that no row of $MM\8$ is zero, thus $M$ has no null rows.
    
    For the converse, we have that if $(MM\8)_{a,b}\neq 0$ then for $\sigma=(ab)$ we have $R_a\not\perp R_b$ by hypothesis and $R_x\not\perp R_x$ since $R_x\neq 0$. 
  
    Putting $D$ in the product $MDM\8$ does not change the orthogonality relations, whence the result is still valid.
\end{proof}

\begin{theorem}
    Let $M$ be a $k\times l$ nonnegative matrix which possesses no null row or column, and such that for permutations matrices $P,Q$ we have
    $$M=P\veto{M_1&&&\\&M_2&&\\&&\ddots&\\&&&M_p}Q,$$
    where $M_r$ has positive entries.
    
    Then, there is a $k\times k$ positive diagonal $D$ and a $l\times l$ positive diagonal $E$ such that for $N:=DME$, both $NN\8$ and $N\8N$ are doubly stochastic.
\end{theorem}
\begin{proof}
We will use Corollary~\ref{cor:sinkcor} to recursively find matrices that satisfy partially what we wanted. Then, we will examine the sub-sequences of these matrices and conclude that there is one of them that converges, by Bolzano-Weierstrass. The limit will be the matrices we wanted.

By Lemma~\ref{lem:17} and Corollary~\ref{cor:sinkcor}, there is a $k\times k$ diagonal $D_1$, for which $D_1MM\8D_1$ is doubly stochastic. Applying this again, there is a $l\times l$ diagonal $E_1$ for which $E_1M\8D_1^2ME_1$ is doubly stochastic. An induction argument gives us $D_i,E_i$ for which
\begin{equation}\label{d.s. 1}
    D_{i+1}ME_i^2M\8D_{i+1}\one=\one;\\
\end{equation}
\begin{equation}\label{d.s. 2}
    E_iM\8D_i^2ME_i\one=\one.
\end{equation}

For each indecomposable block $M_r$ Equation \eqref{d.s. 1} and Equation \eqref{d.s. 2} still hold, with $M_r$ replacing $M$ and the appropriate diagonal block of $D$s and $E$s replacing the whole diagonal. So we can prove the result for each block independently and conclude that it is valid for the whole matrix $M$. For the sake of simplicity, we will maintain the notation and suppose that $M$ has only positive entries.

Denoting $d_{i,x}:=(D_i)_{x,x}$ and $e_{i,y}:=(E_i)_{y,y}$, Equation \ref{d.s. 2} is equivalent to

    $$e_{i,b}\sum_{x=1}^m\sum_{y=1}^n M_{x,b}d_{i,x}^2 M_{x,y}e_{i,y}=1;\qquad b=1,2,\dots, l.$$

We will now prove that each sequence $d_{i,x},e_{i,y}$ is bounded above and below by nonzero constants, which implies that $D_i,E_i$ converge for some sub-sequence to a positive diagonal.

Equation (\ref{d.s. 2}) gives us that
$$e_{i,b}=\dfrac{1}{\sum_{x,y}M_{x,b}d_{i,x}^2M_{x,y}e_{i,y}} \leq\dfrac{1}{d_{i,a}^2M^2_{a,b}e_{i,b}},$$
from where, calling $M:=\max_{x,y}\{M_{x,y}\}$ and $m:=\min_{x,y}\{M_{x,y}\}$, we have:
\begin{equation}\label{ineq. 1}
d_{i,x}e_{i,y}\leq\frac{1}{m},\qquad \forall x, y.
\end{equation}

Equation (\ref{d.s. 2}) and inequality (\ref{ineq. 2}) imply
\begin{equation}\label{ineq. 2}
e_{i,b}=\dfrac{1}{\sum_{x,y}M_{x,b}d_{i,x}^2M_{x,y}e_{i,y}} \geq\dfrac{ml}{M^2\sum_{x}d_{i,x}}.
\end{equation}

Taking sub-sequences we may suppose that all $d_{i,x},e_{i,y}$ either converge (possibly to zero) or go to infinity. Inequality (\ref{ineq. 1}) assures us that if some $d_{i,x}$ goes to infinity then all $e_{i,y}$ go to zero; and if some $e_{i,y}$ goes to infinity then all $d_{i,x}$ go to zero. 

Inequality (\ref{ineq. 2}) implies that if some $e_{i,y}$ goes to zero then some $d_{i,x}$ goes to infinity.

Thus, if no $d_{i,x}$ or $e_{i,y}$ goes to zero, then no sequence goes to infinity and we are done. Suppose otherwise. We have now two cases: for some $x$, $d_{i,x}\mt 0$ or there is $y$ for which $e_{i,y}\mt 0$. We will consider the latter. Equation (\ref{d.s. 2}), will imply, by the above argument, that $e_{i,y}\mt 0$ for all $y$.

Taking sub-sequences we can suppose that each of the following limits exists:
$$\frac{e_{i,y'}}{e_{i,y}}\mt L_{y',y}\in[0,\infty].$$

We will denote by $y\ll y'$ when $L_{y',y}=\infty$. This relation is transitive and anti-symmetrical, hence is a partial order on the values of $y$, and we have some maximal element, which we will call $c$.

Now, call $\tilde{E}_i:=e_{i,c}\1E_i$ and $\tilde{D}_i:=e_{i,c}D_i$. As all $e_{i,y}\mt 0$, by taking some sub-sequence, we can suppose that
$$\frac{e_{i+1,c}}{e_{i,c}}\mt 0.$$

By Equation (\ref{d.s. 1}) we conclude that 
$$
    \tilde{D}_{i+1}M\tilde{E}_i^2M\8\tilde{D}_{i+1}\one=\prt{\frac{e_{i+1,c}}{e_{i,c}}}^2\one;
$$
and
$$
    \tilde{E}_iM\8\tilde{D}_i^2M\tilde{E}_i\one=\one.
$$

As Equation (\ref{d.s. 2}) is still valid, and $\tilde{E}_{c,c}=1$, we will have that no entry of $\tilde{E}_i$ goes to zero. Also, since $y\not\ll c$ for all $y$, we have $(\tilde{E}_i)_{y,y}\not\mt\infty$. Whence, $\tilde{E}_i\mt \tilde{E}$, for some sub-sequence. As $\tilde{E}_iM\8 \tilde{D}_i$ converges (for some sub-sequence) due to Inequality (\ref{ineq. 1}) and the fact that $E_{c,c}=1$, we have $\tilde{D}_i\mt \tilde{D}$. This implies $\tilde{D}M\tilde{E}^2M\8\tilde{D}\one=0$, which contradicts $\tilde{E}M\8\tilde{D}M^2\tilde{E}\one=\one$.

The case $d_{i,x} \mt 0$ for some $x$ is analogous. We change the role of $d_{i,x}$ and $e_{i,y}$ in Inequality (\ref{ineq. 2}) and use Equation (\ref{d.s. 1}) instead of Equation (\ref{d.s. 2}), and the argument follows similarly.
\end{proof}

The theorem below follows immediately.

\begin{theorem}
    Suppose that $G$ and $H$ are positive-weighted graphs. The following are equivalent
    \begin{enumerate}[(i)]
        \item There is a nonzero nonnegative matrix $M$ such that the graph $Z_M$ with adjacency matrix
    $$\veto{&M\\M\8 &}$$
    has as connected components complete bipartite graphs, and $A_G M = M A_H$;
    \item there are pseudo-equitable partitions $(u,\pi)$ and $(w,\sigma)$ of $G$ and $H$ respectively such that $A_{\wt{G/(u,\pi_1)}} = A_{\wt{H/(w,\sigma)}}$.
    \end{enumerate}
    In the case $(i)\implies (ii)$, the classes and isomorphism are defined by the connected components of $Z_M$. \qed
\end{theorem}

\section{Open problems}

There exist polynomial time algorithms to decide whether two graphs are fractionally isomorphic and whether two graphs admit isomorphic quotients with respect to their coarsest equitable partitions. An obvious question is whether there is an efficient algorithm to decide with two graphs admit isomorphic symmetrized quotients.

Our original motivation to look at symmetrized quotients is the connection with quantum walks shown in Section~\ref{sec:quantum}. The main problem we would like to solve is that of finding the smallest graph whose symmetrized quotient matrix is a given one. We warn this is a hard problem, as evidenced by the efforts in \cite{kay2018perfect}.

Fractional isomorphism can be characterized by means of tree homomorphism count. Is there an analogous theory for isomorphic quotients (combinatorial or symmetrized)?

\subsection*{Acknowledgements}

Authors thank Thomás J. Spier, Chris Godsil and Aida Abiad for conversations about the topic of this paper.

Frederico Cançado acknowledges support from CAPES. Gabriel Coutinho acknowledges support from CNPq and FAPEMIG.

\bibliographystyle{plain}
\bibliography{bib}

\end{document}